\documentclass[letterpaper, 10 pt, conference]{ieeeconf}  

\IEEEoverridecommandlockouts                              
\usepackage{graphics} 
\usepackage{epsfig} 

\usepackage{mathptmx} 
\DeclareMathAlphabet{\mathcal}{OMS}{cmsy}{m}{n}
\SetMathAlphabet{\mathcal}{bold}{OMS}{cmsy}{b}{n}

\usepackage{times} 

\usepackage{amsmath} 
\usepackage{amssymb}  
\usepackage{algorithm}
\usepackage{algpseudocode}
\usepackage{mathtools}

\usepackage{ntheorem}
\newtheorem{assumption}{Assumption}
\newtheorem{remark}{Remark}
\newtheorem{problem}{Problem}
\newtheorem{theorem}{Theorem}
\newtheorem{lemma}{Lemma}
\newtheorem{corollary}{Corollary}[lemma]

\usepackage{cite}
\usepackage{stfloats}

\allowdisplaybreaks[2]

\newcommand{\comment}[1]{}
\newcommand{\st}{\mathrm{s.t.}}

\title{\LARGE \bf Generalizing Output-Feedback Covariance Steering to Incorporate Non-Orthogonal Estimation Errors}

\author{Daniel C. Qi and Kenshiro Oguri
\thanks{This work was supported by the U.S. Air Force Office of Scientific Research through grant FA9550-23-1-0512.}
\thanks{D. C. Qi is a PhD student with the School of Aeronautics and Astronautics, Purdue University, West Lafayette, Indiana, 47907, USA. {\tt\small qi85@purdue.edu}}
\thanks{K. Oguri is an Assistant Professor with the School of Aeronautics and Astronautics, Purdue University, West Lafayette, Indiana, 47907, USA.}
}

\begin{document}

\maketitle
\thispagestyle{empty}
\pagestyle{empty}

\begin{abstract}

This paper addresses the problem of steering a state distribution over a finite horizon in discrete time with output feedback. The incorporation of output feedback introduces additional challenges arising from the statistical coupling between the true state distribution and the corresponding filtered state distribution. In particular, this paper extends existing covariance steering formulations to scenarios in which estimation errors are not orthogonal to the state estimates. This paper presents a sequential convex programming framework with rank constraints to solve this more general covariance steering with output feedback problem. The proposed approach is validated through numerical examples and Monte Carlo simulations, including cases with non-orthogonal estimation errors that prior techniques cannot address.

\end{abstract}

\section{Introduction}
This paper addresses the problem of covariance steering (CS) with output feedback in discrete time over a finite horizon. The concept of CS arises from the idea of determining a control policy over a distribution of trajectories rather than optimizing a single trajectory \cite{Hotz-CS, Collins-CS}. Both continuous \cite{Chen-CS} and discrete versions of this problem exist \cite{Bakolas-CS}. Variations of this CS problem have also been developed to address practical considerations, including data-driven CS \cite{Pilipovsky-data-driven-CS}, CS with sparse control inputs \cite{Naoya-Handsoff-CS}, square-root covariance steering \cite{Naoya-Sqrt-CS}, and CS with output feedback \cite{Pilipovsky-CS-Output-Feedback, Oguri-Safe-Autonomy}. Depending on the formulation, the problem can be solved with differential equations \cite{Chen-CS}, semidefinite programming \cite{Liu-CS, Ridderhof-CS}, or sequential convex programming \cite{Naoya-Sqrt-CS}. CS typically assumes Gaussian distributions, and the policy optimization problem reduces to steering the mean and covariance. Although theoretically valid for linear systems, CS has also yielded effective solutions to nonlinear problems through approximation and linearization \cite{Ridderhof-NL-CS, Oguri-Safe-Autonomy, Naoya-CS-JGCD, Qi-min-NL-Journal, Frassinella-Cov-steering, Nuccio-Cov-steering}.

This paper focuses on covariance steering with output feedback (CS-OF). Realistic engineering systems do not have access to the true state, so control actions are performed using the estimated state. Likewise, CS-OF analyzes how the true distribution evolves in the presence of measurement errors. As a natural choice for linear stochastic systems, existing CS-OF algorithms assume the Kalman filter in the formulation. In particular, previous CS-OF formulations assume that the statistics of the state estimate and the estimate error satisfy the orthogonality principle \cite{Chen-CS-Output-Feedback, Ridderhof-CS, Pilipovsky-CS-Output-Feedback, Oguri-Safe-Autonomy}: given a true state $X$ and an estimate $\hat{X}$, the estimation error is said to be orthogonal if 
\begin{equation}
   \mathbb{E}[\hat{X}(X - \hat{X}) ^\top] = 0 
\end{equation}

These approaches can then treat the filter estimation error and the estimated state as separate, uncorrelated stochastic processes, and combine them to compute the distribution of the true state. However, in many engineering scenarios, the estimation error may still be correlated with the estimate. Consider the following three scenarios: 
\begin{enumerate}
    \item Engineers conservatively initialize the state estimate covariance for improved filter convergence \cite{Stat-OD}.
    \item A deliberate suboptimal Kalman gain is used for filter robustness in nonlinear applications \cite{Zanetti-UW}. 
    \item A pre-determined trajectory path is created first, and a filter is propagated along this trajectory for performance and sensitivity analysis \cite{Woffinden-LinCov}. 
\end{enumerate}
For Scenario 1, the filter’s covariance is artificially modified to not align with the actual statistics. In Scenario 2, using a suboptimal gain violates the estimator orthogonality condition, which Kalman filter theory guarantees only when the optimal gain is used \cite{Chui-KalmanTextbook}. Lastly, conventional filter analyses assess estimator performance over an unknown true state. But in Scenario 3, this perspective is changed: the true trajectory is deterministic, and the estimator response is examined relative to this prescribed path. Such scenarios fundamentally alter the statistical relationships, and consequently, existing CS-OF approaches are not directly applicable in these cases of practical relevance. 

The approach proposed in this paper relaxes the orthogonality assumption by modeling the true state and state estimate as correlated stochastic processes. The contributions are threefold: (i) the CS-OF problem is generalized to accommodate non-orthogonal estimation errors, (ii) the proposed framework is shown to solve a class of problems that is a superset of those addressed by existing formulations, and (iii) an efficient rank-constrained solution approach is developed for the generalized CO-OF problem and validated through numerical examples.

\section{Problem Statement}
\subsection{Notation}
$\mathbb{E}[\cdot]$ denotes the expectation, $\mathrm{Cov}(\cdot)$ denotes covariance. $\mathcal{N}(\boldsymbol{\mu},P)$ denotes a normal distribution with mean $\mu$ and covariance $P$. $\mathbb{R}, \mathbb{R}^n, \mathbb{R}^{n\times m}$ denotes the sets of real numbers, $n$-dimensional real-valued vector, and $n\times m$ real-valued matrix respectively. $I_n$ denotes an identity matrix of size $n$. The symbol $\succeq$ denotes matrix inequality between symmetric matrices. $\mathbb{S}^n_+$ denotes the set of $n\times n$ symmetric positive semidefinite matrices. $\mathrm{tr}(\cdot)$ denotes the trace. The notation $\mathbb{Z}_{a:b}$ denotes the set of integers between and including $a,b$.

\subsection{Definitions and Assumptions}
Let the time horizon be discretized into $N$ nodes, forming $N-1$ segments. Let $k = 0,1,\ldots,N-1$ denote the time index. The vectors $\boldsymbol{x}_k \in \mathbb{R}^{n_x}$, $\hat{\boldsymbol{x}}_k \in \mathbb{R}^{n_x}$, and $\hat{\boldsymbol{x}}^-_k \in \mathbb{R}^{n_x}$ denote the true state, filtered state estimate, and a priori state estimate at time step $k$, respectively. The problem of CS-OF treats each of these quantities as a random vector, and thus each has associated statistics. The mean of the true state vector at $k$ is denoted as $\boldsymbol{\mu}_k = \mathbb{E}[\boldsymbol{x}_k]$. The covariances of the true state and the filtered state at $k$ are then given by $P_k = \mathrm{Cov}(\boldsymbol{x}_k)$, $\hat{P}_k = \mathrm{Cov}(\hat{\boldsymbol{x}}_k)$, and $\hat{P}^-_k = \mathrm{Cov}(\hat{\boldsymbol{x}}^-_k)$.
\comment{
\begin{equation}
    \begin{aligned}
P_k &= \mathrm{Cov}(\boldsymbol{x}_k) = \mathbb{E}[(\boldsymbol{x}_k-\mathbb{E}[\boldsymbol{x}_k])(\boldsymbol{x}_k-\mathbb{E}[\boldsymbol{x}_k])^\top] \\
\hat{P}_k &= \mathrm{Cov}(\hat{\boldsymbol{x}}_k) = \mathbb{E}[(\hat{\boldsymbol{x}}_k-\mathbb{E}[\hat{\boldsymbol{x}}_k])(\hat{\boldsymbol{x}}_k-\mathbb{E}[\hat{\boldsymbol{x}}_k])^\top] \\
\hat{P}^-_k &= \mathrm{Cov}(\hat{\boldsymbol{x}}^-_k) = \mathbb{E}[(\hat{\boldsymbol{x}}^-_k-\mathbb{E}[\hat{\boldsymbol{x}}^-_k])(\hat{\boldsymbol{x}}^-_k-\mathbb{E}[\hat{\boldsymbol{x}}^-_k])^\top] \\
    \end{aligned}
\end{equation}
}
Regarding the state statistics and filtering, this paper makes the following assumptions:
\begin{assumption}\label{assump: states are Gaussian}
$x_0$ and $\hat{x}^-_0$ are jointly Gaussian.
\end{assumption}
This implies that the joint distribution exists. Also, it implies that their joint and marginal distributions are Gaussian and are uniquely determined by their mean and covariance \cite{Maybeck-StochasticTextbook}.

\begin{assumption}\label{assump: filter first}
The estimator is designed prior to the control optimization problem.
\end{assumption}
A primary motivation for CS is the offline synthesis of feedback control policies \cite{Ridderhof-CS, Oguri-Safe-Autonomy, Naoya-CS-JGCD, Rapakoulias-CS-CC}. Consistent with existing CS literature, Assumption~\ref{assump: filter first} treats the estimator structure and gains as prescribed inputs to the control optimization problem rather than as optimization variables. But unlike the aforementioned existing CS-OF studies, this assumption does not require the estimator to be optimal and makes no assumptions about the statistical relationships.

\begin{assumption}\label{assump: filter unbiased}
The filter's estimate is unbiased.
\end{assumption}
Most estimators satisfy this assumption.

\begin{assumption}\label{assump: Pk Invert}
$\hat{P}_k$ is invertible.
\end{assumption}
Real-valued covariance matrices are guaranteed to be positive semidefinite. However, this paper assumes that the covariance of the state estimate is strictly positive definite and therefore invertible. This assumption is reasonable, as most engineering systems have nondegenerate covariance matrices. 

\subsection{Covariance Steering with Output Feedback}
The goal of CS-OF is to control the true state $\boldsymbol{x}_k$ using a policy that has the state estimate $\hat{\boldsymbol{x}}_k$ as an input. 
The dynamics of the true state are written as a discrete-time linear system with independent Gaussian noise:
\begin{equation} \label{eq: true state dynamics}
\begin{aligned}
    \boldsymbol{x}_{k+1} &= A_k \boldsymbol{x}_k + B_k \boldsymbol{u}_k + G_k\boldsymbol{w}_k
\end{aligned}
\end{equation}
where $\boldsymbol{u}_k \in \mathbb{R}^{n_u}$, $A_k \in \mathbb{R}^{n_x \times n_x}$, $B_k \in \mathbb{R}^{n_x \times n_u}$, $G_k \in \mathbb{R}^{n_x \times n_w}$ are the system matrices, and Gaussian process noise $\boldsymbol{w}_k\sim\mathcal{N}(0,I_{n_w})$. The measurement model is given by
\begin{equation} \label{eq: msr model}
    \boldsymbol{y}_k = H_k \boldsymbol{x}_k + D_k \boldsymbol{v}_k 
\end{equation}
where $\boldsymbol{y}_k \in \mathbb{R}^{n_y}$, $H_k \in \mathbb{R}^{n_y \times n_x}$ is the measurement mapping matrix, and Gaussian measurement noise $D_k\in \mathbb{R}^{n_y \times n_v}$, $\boldsymbol{v}_k \sim \mathcal{N}(0,I_{n_v})$. The estimation uses a linear gain update $L_k  \in \mathbb{R}^{n_x \times n_y}$ followed by the a priori estimate for the next timestep: 
\begin{subequations}\label{eq: filter updates}
\begin{align}
\hat{\boldsymbol{x}}_k &= \hat{\boldsymbol{x}}^-_k + L_k(\boldsymbol{y}_k - H_k \hat{\boldsymbol{x}}^-_k) 
\label{eq: filter update} \\
\hat{\boldsymbol{x}}^-_{k+1} &= A_k\hat{\boldsymbol{x}}_k + B_k\boldsymbol{u}_k
\label{eq: filter state dynamics}
\end{align}
\end{subequations}
Initial statistics and terminal mean and covariance conditions $\boldsymbol{\mu}_{f}$ and $P_f$ are applied:
\begin{equation}\label{eq: boundary conditions, first}
\begin{aligned}
&\boldsymbol{x}_{0} \sim \mathcal{N}(\boldsymbol{\mu}_{0},P_{0})
&\hat{\boldsymbol{x}}^-_{0} \sim \mathcal{N}(\boldsymbol{\mu}_{0},\hat{P}^-_{0}) \\
&\mathbb{E}[\boldsymbol{x}_{N-1}] = \boldsymbol{\mu}_{f}
&\mathrm{Cov}(\boldsymbol{x}_{N-1}) \preceq P_f\\
\end{aligned} 
\end{equation}
A quadratic control cost $J$ is assumed. The stochastic optimal control problem of CS-OF is then presented in Problem~\ref{prob: stochastic CS-OF}:
\begin{problem}\label{prob: stochastic CS-OF}
Covariance Steering with Output-Feedback:
\begin{subequations}
\begin{align}
\min_{\boldsymbol{x}_k,\boldsymbol{u}_k}
\quad
&
J =  \sum_{k = 0}^{N-2} \mathbb{E}[\boldsymbol{u}_k^\top R_k \boldsymbol{u}_k] \label{eq: stochastic CS-OF, obj}
\\
\st \quad
&\mathrm{Eqs.~\eqref{eq: true state dynamics},\eqref{eq: msr model},\eqref{eq: filter updates}}, \;\forall k \in \mathbb{Z}_{0:N-2}\\
&\mathrm{Eq.~\eqref{eq: boundary conditions, first}}
\end{align} 
\end{subequations}
\end{problem}
with control weight matrix $R_k \succ0$. Eq.~\eqref{eq: stochastic CS-OF, obj} can also have a state quadratic cost \cite{Liu-CS, Pilipovsky-CS-Output-Feedback}, but this term is omitted in this paper for concision. As per Assumption~\ref{assump: filter first}, $L_k$ and $H_k$ are known prior to the optimization, and thus not a part of the optimization variables. In many CS-OF works, the control policy takes the form of a deterministic feedforward term $\bar{\boldsymbol{u}}_k$ with state feedback \cite{Pilipovsky-CS-Output-Feedback, Naoya-CS-JGCD}:
\begin{equation} \label{eq: control policy}
    \boldsymbol{u}_k = \bar{\boldsymbol{u}}_k + K_k (\hat{\boldsymbol{x}}_k - \mu_k)
\end{equation}
where it can be seen that $\mathbb{E}[\boldsymbol{u}_k] = \bar{\boldsymbol{u}}_k$ and $\mathrm{Cov}(\boldsymbol{u}_k) = K_k \hat{P}_kK_k^\top$. With this control policy, the cost function can be rewritten:
\begin{equation}\label{eq: J quad cost}
    J = \sum_{k=0}^{N-2}
\bar{\boldsymbol{u}}_k^\top R_k \bar{\boldsymbol{u}}_k
+ \mathrm{tr}(R_k K_k \hat{P}_kK_k^\top)
\end{equation}

\subsubsection{Kalman Gain}
Typically, the Kalman filter is used as the linear estimator for CS-OF problems \cite{Pilipovsky-CS-Output-Feedback, Chen-CS-Output-Feedback, Naoya-CS-JGCD, Oguri-Safe-Autonomy}.
The Kalman gain $L_{\text{kf},k}$ is computed as
\begin{equation}\label{eq: optimal kalman gain}
    L_{\text{kf},k} = \widetilde{P}^-_k H_k^\top(H_k\widetilde{P}^-_kH_k^\top + D_k D_k^\top)^{-1}
\end{equation}
Notably, Kalman gains produces an orthogonal estimate \cite{Chui-KalmanTextbook}.

\subsubsection{Non-Orthogonal Gain}
Underweighting is a practice that deliberately chooses a suboptimal gain to improve robustness to noise, outliers, or modeling errors \cite{Zanetti-UW}. Consider a suboptimal, underweighted Kalman gain ${L}_{\mathrm{uw},k}$ in Eq.~\eqref{eq: uw Kalman gain}:
\begin{equation}\label{eq: uw Kalman gain}
    {L}_{\mathrm{uw},k} = \widetilde{P}^-_k H_k^\top
    ( p^{-1}H_k\widetilde{P}^-_kH_k^\top + D_k D_k^\top)^{-1}
\end{equation}
where $p \in (0,1]$ is the underweighting factor. As $p \rightarrow 1$, the underweight gain approaches the optimal Kalman gain, and as $p \rightarrow 0$, the measurement update is increasingly downweighted. Since the gain is now suboptimal, this underweighted Kalman gain produces a non-orthogonal estimate.

\subsubsection{Filter Covariance Updates With Arbitrary Gain}
The Joseph's form for the filter covariance update can be used regardless of whether $L_{\text{kf},k}$ or ${L}_{\mathrm{uw},k}$ is computed: 
\begin{equation}\label{eq: Kalman Covariance Update}
 \widetilde{P}_k = (I_{n_x} - L_k H_k)\widetilde{P}^-_k (I_{n_x} - L_k H_k)^\top + L_k D_k D_k^\top L_k^\top
\end{equation}
where $L_k$ denotes \emph{any} gain. The time update is
\begin{equation}
    \widetilde{P}^-_{k+1} = A_k \widetilde{P}_k A_k^\top + G_kG_k^\top,
\end{equation}

\section{Output-Feedback Covariance Steering with Non-Orthogonal State Estimate}
\subsection{Covariance Steering with Augmented State Statistics}
This paper adopts an augmented-state approach, similar to \cite{Chen-CS-Output-Feedback, Maybeck-StochasticTextbook, Geller-LinCov}, to track statistics. Define
\begin{equation}
    \mathbf{X}^-_k \triangleq 
    \begin{bmatrix}
        \boldsymbol{x}_k \\ \hat{\boldsymbol{x}}^-_k
    \end{bmatrix}
\qquad
    \mathbf{X}_k \triangleq 
    \begin{bmatrix}
        \boldsymbol{x}_k \\ \hat{\boldsymbol{x}}_k
    \end{bmatrix}
\end{equation}
where $\mathbf{X}^-_k$ is referred to as the \emph{a priori augmented state} and $\mathbf{X}_k$ as the \emph{a posteriori augmented state}. For the covariance evolution, define the following notation:
\begin{equation}
    \begin{aligned}
\mathbf{P}^-_k = \mathrm{Cov}(\mathbf{X}^-_k)
\qquad
\mathbf{P}_k = \mathrm{Cov}(\mathbf{X}_k)
    \end{aligned}
\end{equation}
\comment{
\begin{equation}
    \begin{aligned}
\mathbf{P}^-_k &= \mathrm{Cov}(\mathbf{X}^-_k) = \mathbb{E}[(\mathbf{X}^-_k-\mathbb{E}[\mathbf{X}^-_k])(\mathbf{X}^-_k-\mathbb{E}[\mathbf{X}^-_k])^\top] \\
\mathbf{P}_k &= \mathrm{Cov}(\mathbf{X}_k) =\mathbb{E}[(\mathbf{X}_k-\mathbb{E}[\mathbf{X}_k])(\mathbf{X}_k-\mathbb{E}[\mathbf{X}_k])^\top] 
    \end{aligned}
\end{equation}
}
and it can be seen that the submatrices correspond to the covariances of the original states.
\begin{equation}\label{eq: Phat block matrix}
\mathbf{P}^-_k  =
\begin{bmatrix}
    P_k & (\Sigma^-_k)^\top \\
    \Sigma^-_k & \hat{P}^-_k
\end{bmatrix}
\qquad
    \mathbf{P}_k  =
\begin{bmatrix}
    P_k & \Sigma_k^\top \\
    \Sigma_k & \hat{P}_k
\end{bmatrix}
\end{equation}
where $\Sigma^-_k = \mathrm{Cov}(\hat{\boldsymbol{x}}^-_k,\boldsymbol{x}_k)$ and $\Sigma_k = \mathrm{Cov}(\hat{\boldsymbol{x}}_k,\boldsymbol{x}_k)$.

\begin{lemma}[Augmented Mean/Covariance Evolution]\label{lemma: augmented stat evolution}
Under Assumptions~1--3 and the control policy in Eq.~\eqref{eq: control policy}, the evolution of the augmented statistics satisfy
\begin{subequations}\label{eq: augmented stat evolution}
\begin{align}
\boldsymbol{\mu}_{k+1} &= A_k\boldsymbol{\mu}_{k} + B_k\bar{\boldsymbol{u}}_k  
\label{eq: mean evolution} \\
\mathbf{P}_k&= \Phi_k\mathbf{P}^-_k\Phi_k^\top + \mathbf{L}_k
\label{eq: covariance filter update}\\
\mathbf{P}^-_{k+1} &= 
\left(\mathbf{A}_k + \mathbf{B}_k
\mathbf{K}_k\right)\mathbf{P}_k\left(\mathbf{A}_k + \mathbf{B}_k \mathbf{K}_k\right)^\top + \mathbf{Q}_k
\label{eq: covariance control update} 
\end{align}
\end{subequations}
with defined matrices
\begin{equation}
    \Phi_k \triangleq
    \begin{bsmallmatrix}
    I_{n_x} & 0 \\
    L_kH_k & I_{n_x} - L_k H_k
    \end{bsmallmatrix}
    ,\
    \mathbf{L}_k \triangleq 
    \begin{bsmallmatrix}
    0 & 0 \\
    0 & L_k D_k D_k^\top L_k^\top
    \end{bsmallmatrix}
\end{equation}
and
\begin{equation}
\begin{aligned}
    \mathbf{A}_k \triangleq
    \begin{bsmallmatrix}
    A_k & 0 \\
    0 & A_k
    \end{bsmallmatrix}
,\
    \mathbf{B}_k  \triangleq
    \begin{bsmallmatrix}
    B_k & 0 \\
    0 & B_k
    \end{bsmallmatrix}
,\
   \mathbf{K}_k \triangleq
    \begin{bsmallmatrix}
    0 & K_k \\
    0 & K_k
    \end{bsmallmatrix}
,\
    \mathbf{Q}_k \triangleq
    \begin{bsmallmatrix}
    G_kG_k^\top & 0 \\
    0 & 0
    \end{bsmallmatrix}
\end{aligned}
\end{equation}

\end{lemma}

\begin{proof}
Using Eqs.~\eqref{eq: msr model} and \eqref{eq: filter update}:
\begin{equation} \label{eq: augmented filter update}
    \begin{bsmallmatrix}
        \boldsymbol{x}_k \\ \hat{\boldsymbol{x}}_k
    \end{bsmallmatrix}
    =
    \begin{bsmallmatrix}
    I_{n_x} & 0 \\
    L_kH_k & I_{n_x} - L_k H_k
    \end{bsmallmatrix}
    \begin{bsmallmatrix}
        \boldsymbol{x}_k \\ \hat{\boldsymbol{x}}^-_k
    \end{bsmallmatrix}
    +
    \begin{bsmallmatrix}
        0 \\ L_k D_k
    \end{bsmallmatrix}
    \boldsymbol{v}_k
\end{equation}
Using Eqs.~\eqref{eq: true state dynamics}, \eqref{eq: filter state dynamics}, and \eqref{eq: control policy}:
\begin{equation} \label{eq: augmented control update}
\begin{aligned}
\begin{bsmallmatrix}
    \boldsymbol{x}_{k+1} \\ \hat{\boldsymbol{x}}^-_{k+1}
\end{bsmallmatrix}
=
    \begin{bsmallmatrix}
    A_k & B_k K_k \\
    0 & A_k + B_kK_k
    \end{bsmallmatrix}
    \begin{bsmallmatrix}
        \boldsymbol{x}_k \\ \hat{\boldsymbol{x}}_k
    \end{bsmallmatrix}
    +
    \begin{bsmallmatrix}
        B_k \\ B_k
    \end{bsmallmatrix}
    (\bar{u}_k - K_k \boldsymbol{\mu}_k)
        +    
    \begin{bsmallmatrix}
        G_k \\ 0
    \end{bsmallmatrix}\boldsymbol{w}_k
\end{aligned}
\end{equation}

Since Eqs.~\eqref{eq: augmented filter update} and \eqref{eq: augmented control update} are linear in the augmented state, its mean can be computed through linearity of expectations. From Assumption~\ref{assump: filter unbiased}, the estimate is unbiased (i.e., $\mathbb{E}[\boldsymbol{x}_k] = \mathbb{E}[\hat{\boldsymbol{x}}^-_k] = \mathbb{E}[\hat{\boldsymbol{x}}_k] = \boldsymbol{\mu}_k$ for all $k$), so the evolution of the mean can be reduced to Eq~\eqref{eq: mean evolution}. Using the defined matrices and noting $v_k$ and $w_k$ are
independent of the augmented states, the covariance evolution for Eq.~\eqref{eq: augmented filter update} and \eqref{eq: augmented control update} becomes Eqs.~\eqref{eq: covariance filter update} and \eqref{eq: covariance control update} respectively.
\end{proof}
Eq.~\eqref{eq: boundary conditions, first} is changed to reflect the augmented statistics:
\begin{equation} \label{eq: boundary conditions}
\begin{aligned}
\boldsymbol{\mu}_0 = \bar{\boldsymbol{\mu}}_0 
,\quad
\mathbf{P}^-_0 = \bar{\mathbf{P}}^-_0
,\quad
\boldsymbol{\mu}_{N-1} = \boldsymbol{\mu}_f 
,\quad
P_{N-1} \preceq P_f
\end{aligned}
\end{equation}
with initial joint statistics constrained to $\bar{\boldsymbol{\mu}}_0$ and $\bar{\mathbf{P}}^-_0$. Problem~\ref{prob: stochastic CS-OF} can be expressed as Problem~\ref{prob: CS with Aug State} using Lemma~\ref{lemma: augmented stat evolution}:
\begin{problem}\label{prob: CS with Aug State}
CS-OF with Augmented States:
\begin{subequations}
\begin{align}
\min_{
\boldsymbol{\mu}_k, \mathbf{P}^-_k,\mathbf{P}_k,\bar{\boldsymbol{u}}_k}
\quad& 
J = \sum_{k=0}^{N-2}
\bar{\boldsymbol{u}}_k^\top R_k \bar{\boldsymbol{u}}_k
+ \mathrm{tr}(R_k K_k \hat{P}_kK_k^\top)
\\
\st \quad
&\mathrm{Eq.~\eqref{eq: augmented stat evolution}}, \;\forall k \in \mathbb{Z}_{0:N-2}\\
&\mathrm{Eq.~\eqref{eq: boundary conditions}}
\end{align} 
\end{subequations}
\end{problem}

\subsubsection{Relation to Previous Works}
Prior works in CS-OF make the additional orthogonality assumption and under this, the covariance evolution can be written in the form
\begin{equation}\label{eq: CSOF, old version}
    \begin{aligned}
\hat{P}_{k+1} &= (A_k +B_kK_k)\hat{P}_k(A_k +B_kK_k)^\top + L_{k+1}P_{\widetilde{y}^-_{k+1}}L_{k+1}^\top \\
P_k &= \hat{P}_{k} + \widetilde{P}_k
    \end{aligned}
\end{equation}
where $P_{\widetilde{y}^-_k} = H_k \widetilde{P}^-_k H_k^\top + D_k D_k^\top$ is the innovation covariance and initial statistics are computed with $P_0 = \hat{P}^-_0 + \widetilde{P}^-_0$, $\hat{P}_0 = \hat{P}^-_0 + L_0P_{\widetilde{y}^-_0}L_0^\top$ \cite{Ridderhof-CS, Pilipovsky-CS-Output-Feedback, Naoya-CS-JGCD}. With this, the following lemma and remark are made:

\begin{lemma}\label{lemma: CSOF equivalence}
If the orthogonality of the estimator is assumed, Eqs.~\eqref{eq: covariance filter update} and \eqref{eq: covariance control update} are reduced to Eq.~\eqref{eq: CSOF, old version}. 
\end{lemma}
\begin{proof}
See Appendix~\ref{appendix: CSOF equivalence}.
\end{proof}

\begin{remark}
It is important to distinguish the realized estimation error covariance $\mathrm{Cov}(\boldsymbol{x}_k - \hat{\boldsymbol{x}}_k)$ from the error covariance $\widetilde{P}_k$ reported by the filter. The covariance $\mathrm{Cov}(\boldsymbol{x}_k - \hat{\boldsymbol{x}}_k)= P_k + \hat{P}_k - \Sigma_k - \Sigma_k^\top$ holds unconditionally, whereas $\widetilde{P}_k$ depends on the statistical assumptions employed in the filter's design.
\end{remark}

\subsection{Convex Relaxation}
One of the main difficulties in optimizing CS control policies is the bilinearities that appear in Eq.~\eqref{eq: covariance control update} which make the problem nonconvex. This section is inspired by the common substitution techniques in previous CS work to resolve these bilinearities \cite{Liu-CS}.

\begin{lemma}[Convex Relaxation of Lemma~\ref{lemma: augmented stat evolution}]\label{lemma: augmented stat evolution, relaxed}
Under Assumption~\ref{assump: Pk Invert}, define lifting variables
\begin{subequations}\label{eq: relaxation variables}
\begin{align}
U_k &\triangleq K_k \hat{P}_k \label{eq: Uk def}\\
Y_k &\triangleq U_k \hat{P}_k^{-1}U_k^\top (=K_k \hat{P}_k K_k^\top) 
,\quad\in \mathbb{S}^{n_u}_+ \label{eq: Yk def} \\
S_k &\triangleq K_k \Sigma_k \label{eq: Sk def} \\
Z_k &\triangleq \Sigma_k^\top \hat{P}_k^{-1}\Sigma_k 
,\quad\in \mathbb{S}^{n_u}_+ \label{eq: Zk def}
\end{align}
\end{subequations}
and defined matrices
\begin{equation}\label{eq: relaxed matrices}
\begin{aligned}
    \mathbf{Y}_k \triangleq
    \begin{bsmallmatrix}
    Y_k & Y_k \\
    Y_k & Y_k
    \end{bsmallmatrix}
,\
    \mathbf{W}_k  \triangleq
    \begin{bsmallmatrix}
    S_k & U_k \\
    S_k & U_k
    \end{bsmallmatrix}
,\
   \mathbf{S}_k \triangleq
    \begin{bsmallmatrix}
    Y_k & S_k \\
    S_k^\top & Z_k
    \end{bsmallmatrix}
,\
    \mathbf{U}_k  \triangleq
    \begin{bsmallmatrix}
    U_k \\ \Sigma_k^\top
    \end{bsmallmatrix}
\end{aligned}
\end{equation}

Then, a convex relaxation of Eq.~\eqref{eq: covariance control update} is
\begin{subequations}\label{eq: augmented relaxation evolution}
\begin{gather}
\begin{aligned}
    \mathbf{P}^-_{k+1} =& 
\mathbf{A}_k\mathbf{P}_k\mathbf{A}_k^\top
+\mathbf{B}_k\mathbf{Y}_k\mathbf{B}_k^\top \\
&+\mathbf{B}_k\mathbf{W}_k\mathbf{A}_k^\top
+\mathbf{A}_k\mathbf{W}_k^\top\mathbf{B}_k^\top
+ \mathbf{Q}_k
\end{aligned}
\label{eq: covariance control update, relaxed}\\
M_k \triangleq
\begin{bmatrix}
\hat{P}_k & \mathbf{U}_k^\top \\
     \mathbf{U}_k & \mathbf{S}_k
\end{bmatrix}
\succeq 0, 
\qquad
\in
\mathbb{R}^{m \times m}
\label{eq: Sk LMI}
\end{gather}
\end{subequations}
where $m = 2n_x+n_u$. 
\end{lemma}
\begin{proof}
Substituting Eq.~\eqref{eq: Phat block matrix} into Eq.~\eqref{eq: covariance control update} yields
\begin{equation}\label{eq: covariance control update, expanded}
\begin{aligned}
\mathbf{P}^-_{k+1} =&
\mathbf{A}_k \mathbf{P}_{k} \mathbf{A}_k^\top
+ 
\mathbf{B}_k
\begin{bsmallmatrix}
    K_k \hat{P}_k K_k^\top& K_k \hat{P}_k K_k^\top \\
    K_k \hat{P}_k K_k^\top& K_k \hat{P}_k K_k^\top
\end{bsmallmatrix}
\mathbf{B}_k^\top
\\
&+ \mathbf{B}_k
\begin{bsmallmatrix}
    K_k \Sigma_k & K_k \hat{P}_k \\
    K_k \Sigma_k & K_k \hat{P}_k
\end{bsmallmatrix}\mathbf{A}_k^\top
+\mathbf{A}_k
\begin{bsmallmatrix}
    K_k \Sigma_k & K_k \hat{P}_k \\
    K_k \Sigma_k & K_k \hat{P}_k
\end{bsmallmatrix}^\top \mathbf{B}_k^\top
+\mathbf{Q}_k
\end{aligned} 
\end{equation}
Substituting quantities from Eqs.~\eqref{eq: relaxation variables} and \eqref{eq: relaxed matrices} yields Eq.~\eqref{eq: covariance control update, relaxed}. From Assumption~\ref{assump: Pk Invert}, $K_k$ is unique:
\begin{equation}\label{eq: gain recovery}
    K_k = U_k \hat{P}_k^{-1}
\end{equation}
Note that the cross-covariance matrices $\Sigma_k$ are not guaranteed to be invertible, but the value of $K_k$ in Eq.~\eqref{eq: Sk def} must be identical to the value found in Eq.~\eqref{eq: gain recovery}. Substituting,
\begin{equation} \label{eq: Sk non symmetrizes}
    S_k = U_k \hat{P}_k^{-1} \Sigma_k 
    \quad\Leftrightarrow\quad
    S_k - U_k \hat{P}_k^{-1}\Sigma_k = 0
\end{equation}
Noting that $S_k$ is not symmetric, Eq.~\eqref{eq: Sk aug unrelaxed} is introduced to facilitate the subsequent semidefinite formulations. Satisfaction of Eq.~\eqref{eq: Sk aug unrelaxed} also implies satisfaction of Eq.~\eqref{eq: Sk non symmetrizes}.
\begin{equation} \label{eq: Sk aug unrelaxed}
\begin{aligned}
&
\begin{bsmallmatrix}
    Y_k & S_k  \\
    S_k^\top & Z_k 
\end{bsmallmatrix}
-
\begin{bsmallmatrix}
    U_k  \\
    \Sigma_k^\top
\end{bsmallmatrix}
\hat{P}_k^{-1}
\begin{bsmallmatrix}
    U_k \\
    \Sigma_k^\top
\end{bsmallmatrix}^\top
=0 \\
\end{aligned}
\end{equation}
and Eq.~\eqref{eq: Sk aug unrelaxed} is compactly expressed as $\mathbf{S}_k - \mathbf{U}_k\hat{P}_k^{-1}\mathbf{U}_k^\top =0$. Since the expression is still nonconvex, the equality constraint is relaxed with a matrix inequality constraint.
\begin{equation} \label{eq: Sk relaxation}
\mathbf{S}_k - \mathbf{U}_k
\hat{P}_k^{-1}
\mathbf{U}_k^\top  \succeq0
\end{equation}
Using Schur's complement \cite{Zhang-SchurTextbook}, it is equivalent to Eq.~\eqref{eq: Sk LMI}
\end{proof}
These techniques from previous CS works can convert Problem~\ref{prob: CS with Aug State} into a similarly convexified problem:
\begin{problem}\label{prob: relaxed aug state}
Convex Relaxation of Problem~\ref{prob: CS with Aug State}:
\begin{subequations}\label{eq: relaxed aug state}
\begin{align}
\min_{\substack{
\boldsymbol{\mu}_k, \mathbf{P}^-_k,\mathbf{P}_k,\bar{\boldsymbol{u}}_k
\\U_k, S_k, Z_k, Y_k}}
\quad& J = \sum_{k=0}^{N-2}
\bar{\boldsymbol{u}}_k^\top R_k \bar{\boldsymbol{u}}_k
+ \mathrm{tr}(R_k Y_k)
\label{eq: Obj - trace Yk and Zk}
\\
\st \quad 
&\mathrm{Eqs.~\eqref{eq: mean evolution},\eqref{eq: covariance filter update}}, \;\forall k \in \mathbb{Z}_{0:N-2}\\
&\mathrm{Eqs.~\eqref{eq: augmented relaxation evolution}}, \;\forall k \in \mathbb{Z}_{0:N-2}\\
&\mathrm{Eq.~\eqref{eq: boundary conditions}}
\end{align} 
\end{subequations}
\end{problem} 
\begin{remark}
Prior work in the CS literature has shown that this relaxation technique is lossless at optimality under appropriately chosen objective functions, specifically $\mathrm{tr}(Y_k)$ \cite{Liu-CS, Rapakoulias-CS-CC, Pilipovsky-CS-Output-Feedback}. Unfortunately, the corresponding losslessness proofs do not extend directly to the relaxed Problem~\ref{prob: relaxed aug state}. 
\end{remark}

\subsection{Restoring Losslessness via a Rank Constraint}
\begin{corollary}\label{col: Sk LMI equality}
Under Assumption~\ref{assump: Pk Invert}, if $\mathbf{S}_k - \mathbf{U}_k\hat{P}_k^{-1}\mathbf{U}_k^\top = 0$, then $K_k$ is unique and Eq.~\eqref{eq: covariance control update} coincides with Eq.~\eqref{eq: covariance control update, relaxed}, making Lemma~\ref{lemma: augmented stat evolution, relaxed} lossless convexification.
\end{corollary}
\begin{proof}
Eq.~\eqref{eq: gain recovery} shows a unique gain, and expressions in Eq.~\eqref{eq: relaxation variables} are all simultaneously satisfied with Eq.~\eqref{eq: Sk aug unrelaxed}.
\end{proof}

The application of Corollary~\ref{col: Sk LMI equality} motivates imposing the proposed rank constraint on $M_k$. Applying the Guttman rank additivity formula \cite{Zhang-SchurTextbook},
\begin{equation}
    \mathrm{rank}(M_k) = \mathrm{rank}(\hat{P}_k) + \mathrm{rank}(\mathbf{S}_k - \mathbf{U}_k
\hat{P}_k^{-1}
\mathbf{U}_k^\top)
\end{equation}
Under Assumption \ref{assump: Pk Invert}, $\hat{P}_k$ is full rank. Thus, it can be seen that for Corollary~\ref{col: Sk LMI equality}  to be true, $\mathrm{rank}(M_k)$ must be equal to $\mathrm{rank}(\hat{P}_k)=n_x$. Rank constraints are generally nonconvex, so solving it through convex optimization requires sequential convex programming (SCP).

\comment{
\begin{remark}
    The problem can also be addressed with standard linearization and SCP. However, the bilinearities present make the problem particularly challenging from a nonconvex optimization standpoint. In contrast, Problem~\ref{prob: relaxed aug state} removes the bilinearities through relaxation, resulting in a formulation for which several established approaches for rank-constrained optimization are available.
\end{remark}
}

Firstly, trace minimization is a widely used heuristic for obtaining approximate solutions to rank-constrained problems \cite{Sun-IRM}. Because $\mathrm{tr}(S_k) = \mathrm{tr}(Y_k) + \mathrm{tr}(Z_k)$, this heuristic need to account for both trace terms. The term $\mathrm{tr}(Y_k)$ already appears in the original quadratic cost in Eq.~\eqref{eq: J quad cost}, but the trace heuristic should still consider a $\mathrm{tr}(Z_k)$ component. This motivates a small weight $\epsilon_Z$ on $\mathrm{tr}(Z_k)$ to promote the trace heuristic while maintaining the original quadratic cost.
\begin{equation}
    J_Z = J + \epsilon_Z\cdot\sum_{k=0}^{N-2} \mathrm{tr}(Z_k)
\end{equation}

Next, the Iterative Rank Minimization (IRM) algorithm from \cite{Sun-IRM} is used to solve this rank-constrained problem. To summarize IRM: let $\boldsymbol{V}^{(i)}_k$ be the orthonormal eigenvectors corresponding to the $m - n_x$ smallest eigenvalues and $e^{(i)}_k$ be greater than or equal to the $(m - n_x)$-th smallest eigenvalue of $M_k$ at the $i$-th iteration. Since $M_k$ is a symmetric matrix and thus diagonalizable, $\mathrm{rank}(M_k)$ is equal to the number of nonzero eigenvalues, implying $e^{(i)}_k =0$ for rank constraint satisfaction. Calculation of $e^{(i)}_k$ is nonconvex, so \cite{Sun-IRM} approximates it by using the previous iteration's eigenvectors:
\begin{equation}\label{eq: Mk VD relation}
    \begin{aligned}
e_k I_{m - n_x} - (V^{(i-1)}_{k})^\top M_k V^{(i-1)}_{k} \succeq 0,
\qquad
e_k\geq 0     
\end{aligned}
\end{equation}
where $e_k$ is the optimized rank-violation slack variable. To ensure that $e^{(i)}_k\rightarrow0$, the original IRM algorithm \cite{Sun-IRM} enforces a hard constraint $e_k \leq e^{(i-1)}_{k}$ at each $i$-th convex subproblem. It was observed that this led to poor numerical properties for CS-OF problems. Instead, this paper removes the hard constraints, and an augmented Lagrangian approach with a quadratic penalty term \cite{Bertsekas-LM, Oguri-SCVX} is used: 
\begin{equation} \label{eq: quad penalty}
\begin{aligned}
\mathcal{P}(w^{(i)},\lambda^{(i)}_k,e_k) = \lambda^{(i)}_k \cdot e_k + (w^{(i)}/2) \cdot e_k \cdot e_k
\end{aligned}
\end{equation}
After each convex subproblem, $e^{(i)}_k$ is set equal to the $(m - n_x)$-th smallest eigenvalue of $M_k$, and the Lagrange multipliers and weights are updated accordingly:
\begin{equation} \label{eq: LM updates}
\begin{aligned}
    \lambda^{(i+1)}_k = \lambda^{(i)}_k + w^{(i)} e^{(i)}_k,
    \qquad
    w^{(i+1)} = \beta w^{(i)} \quad (\beta > 1)
\end{aligned}
\end{equation}

At the $i$-th iteration, Problem~\ref{prob: irm} is solved.
\begin{problem}\label{prob: irm}
Problem~\ref{prob: relaxed aug state} with rank constraint, single iteration
\begin{subequations}
\begin{align}
\min_{\substack{
\boldsymbol{\mu}_k, \mathbf{P}^-_k,\mathbf{P}_k,\bar{\boldsymbol{u}}_k
\\U_k, S_k, Z_k, Y_k,
e_k
}} 
\quad& \mathcal{L} = J_Z + \sum_{k=0}^{N-2}\mathcal{P}(w^{(i)},\lambda^{(i)}_k,e_k)\\
\st \quad
&\mathrm{Eqs.~\eqref{eq: mean evolution},\eqref{eq: covariance filter update}}
, \;\forall k \in \mathbb{Z}_{0:N-2}
\\
&\mathrm{Eqs.~\eqref{eq: augmented relaxation evolution},\eqref{eq: Mk VD relation}}
, \;\forall k \in \mathbb{Z}_{0:N-2}
\\
&\mathrm{Eq.~\eqref{eq: boundary conditions}}
\end{align} 
\end{subequations}
\end{problem}
The pseudocode for sequentially solving CS with augmented state statistics is given in Algorithm~\ref{algo: IRM algo}. Finally, Theorem~\ref{theorem: irm convergence} presents the main result of the paper.
\begin{algorithm}
\caption{CS-OF with Rank Constraints}
\label{algo: IRM algo}
\begin{algorithmic}
\State \textbf{Input:} $w_0, \beta, \epsilon_1, \epsilon_2$
\State \textbf{Output:} $\bar{\boldsymbol{u}}_k,K_k$
\State \textbf{Initialize:} Solve the relaxed Problem~\ref{prob: relaxed aug state} with trace-heuristic objective $J_Z$ to obtain $V^{(0)}_k$ and $e^{(0)}_k$. Set $i = 0$.
\While{$\max_k ({e^{(i)}_k}) \ge \epsilon_1$ \textbf{or} $|{J^{(i)} - J^{(i-1)}}| \ge \epsilon_2$}
\State Update $i = i + 1$.
\State Solve Problem~\ref{prob: irm}.
\State Update $V^{(i)}_k$ and $e^{(i)}_k$ via eigenvalue decomposition.
\State Update $\lambda^{(i)}, w^{(i)}$ using Eq.~\eqref{eq: LM updates}.
\EndWhile
\State \textbf{Return:} $\bar{\boldsymbol{u}}_k$, $K_k=U_k \hat{P}_k^{-1}$ upon convergence.
\end{algorithmic}
\end{algorithm}

\begin{theorem}\label{theorem: irm convergence}
The solution of Algorithm~\ref{algo: IRM algo} is a local optimum of an $\epsilon_Z$-regularized approximation of Problem~\ref{prob: CS with Aug State}.
\end{theorem}
\begin{proof}
Under the local-convergence conditions of IRM and augmented-Lagrangian SCP~\cite{Sun-IRM, Bertsekas-LM, Oguri-SCVX}, Algorithm~\ref{algo: IRM algo} converges to a local optimum of the regularized rank-constrained formulation and Corollary~\ref{col: Sk LMI equality} is achieved. Since the lifted objective differs from the original objective only by $\epsilon_Z\sum_k\mathrm{tr}(Z_k)$, the resulting solution is an $\epsilon_Z$-regularized local optimum and becomes a local optimum of the original problem when $\epsilon_Z=0$.
\end{proof}

\section{Numerical Results}
A double integrator system with $N=20$ discretized nodes is used as a demonstration. All other simulation parameters are found in Appendix~\ref{appendix: sim params}. 

\subsection{Generalization of Initial Statistics}
This section compares the proposed generalized framework with the previous works in \cite{Pilipovsky-CS-Output-Feedback} and presents scenarios in which the initial statistics cannot be addressed by the previous method. Consider two initial a priori augmented state covariances in Eq.~\eqref{eq: results cases}:
\begin{subequations}\label{eq: results cases}
\begin{align}
\text{Case 1:}\quad
\mathbf{P}^-_0 &= 
\begin{bmatrix}
    \hat{P}^-_0 + \widetilde{P}^-_0&\hat{P}^-_0\\
    \hat{P}^-_0 & \hat{P}^-_0
\end{bmatrix}
\quad
&\text{with}
\quad
{L}_{\mathrm{kf},k}
\label{eq: case1} \\
~
\text{Case 2:}\quad
\mathbf{P}^-_0 &= 
\begin{bmatrix}
    P_0 & P_0\\
    P_0& P_0 + \widetilde{P}^-_0
\end{bmatrix}
\quad
&\text{with}
\quad
{L}_{\mathrm{kf},k}
\label{eq: case2} 
\end{align}
\end{subequations}
where the block matrices correspond to Eq.~\eqref{eq: Phat block matrix}. Case 1 corresponds to the case in which the orthogonality between the state estimate and the estimation error is satisfied, consistent with the scenario considered in \cite{Pilipovsky-CS-Output-Feedback}. Conversely, Case 2 demonstrates a scenario where the error is \emph{orthogonal to the true state}. This is shown in Appendix~\ref{appendix: case 1 and 2 initial covariance}.

The relevance of Case 2 is motivated by its connection to the analysis of control performance about a fixed plant trajectory (see Scenario 3 in Introduction). Prior studies such as \cite{Woffinden-LinCov}, and related work in \cite{Williams-LinCov, Cunningham-LinCov}, examine the impact of control actions arising from measurement feedback errors on a \emph{pre-determined} truth trajectory. As such, the statistics described in Eq.~\eqref{eq: case2} assume that the estimation error is \emph{uncorrelated with the true state}. Appendix~\ref{appendix: MC experiment} provides a physical interpretation of the statistical quantities used in the Monte Carlo experiment. Orthogonality also implies $\hat{P}^-_0 \prec P_0$, so the methods in \cite{Pilipovsky-CS-Output-Feedback} cannot handle, and therefore cannot synthesize control policies for, Case 2 where $P_0 \prec \hat{P}^-_0$. This limitation is shared by related approaches in  \cite{Ridderhof-CS, Chen-CS-Output-Feedback, Oguri-Safe-Autonomy}.

\begin{figure}[!tbh]
\centering\includegraphics[width=0.49\textwidth]{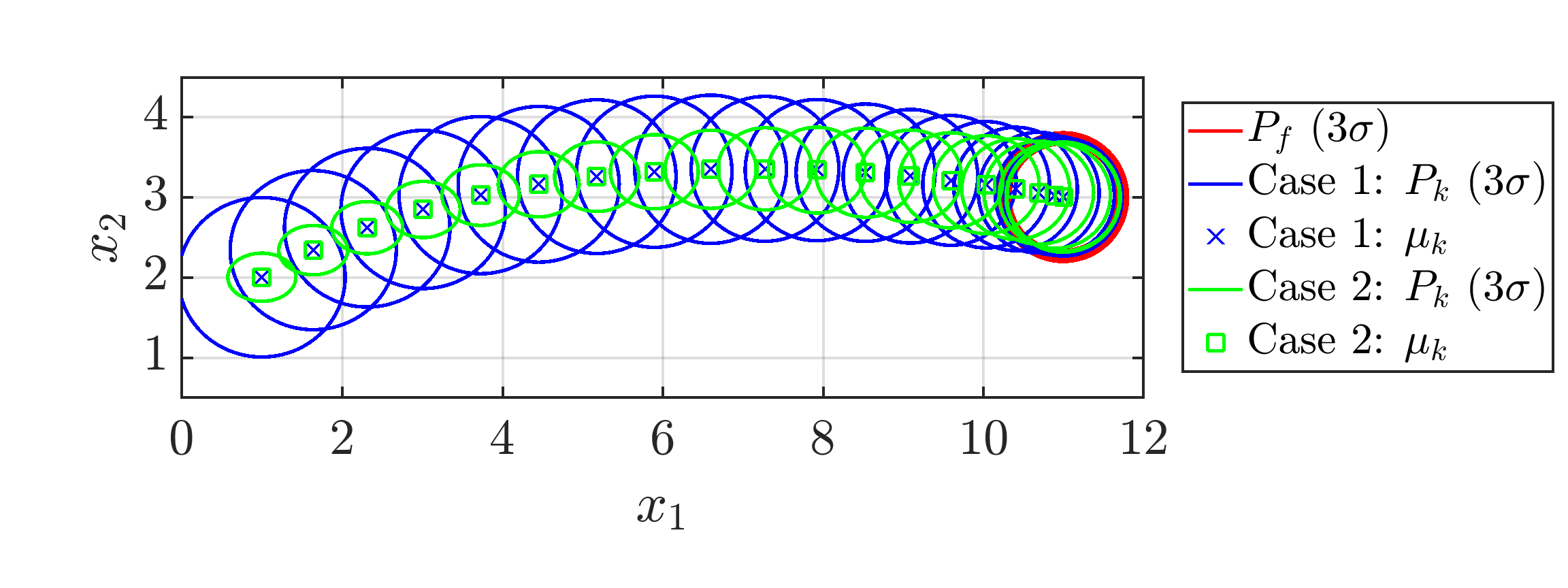}
	\caption{Case 1 \& 2: Trajectory of true state distributions calculated using the proposed method.}
	\label{fig: traj}
\end{figure}

Figure~\ref{fig: traj} shows the optimized distributions of each case. Despite both cases having the same a priori estimate covariance, the behavior of the true covariance differs: for Case 1, the covariance shrinks, while for Case 2, the covariance grows. To validate the proposed methods, the optimized results are compared with those obtained from prior work \cite{Pilipovsky-CS-Output-Feedback} as well as with outcomes from a large-scale Monte Carlo simulation based on the sampling procedure described in Table~\ref{tab: MC_Sampling}. The comparison is presented in Figure~\ref{fig: validation}. Monte Carlo simulations for both Case 1 and Case 2 exhibit statistical consistency with the theoretical predictions, providing additional validation of the proposed framework.

\begin{figure}[!tbh]
\centering\includegraphics[width=0.49\textwidth]{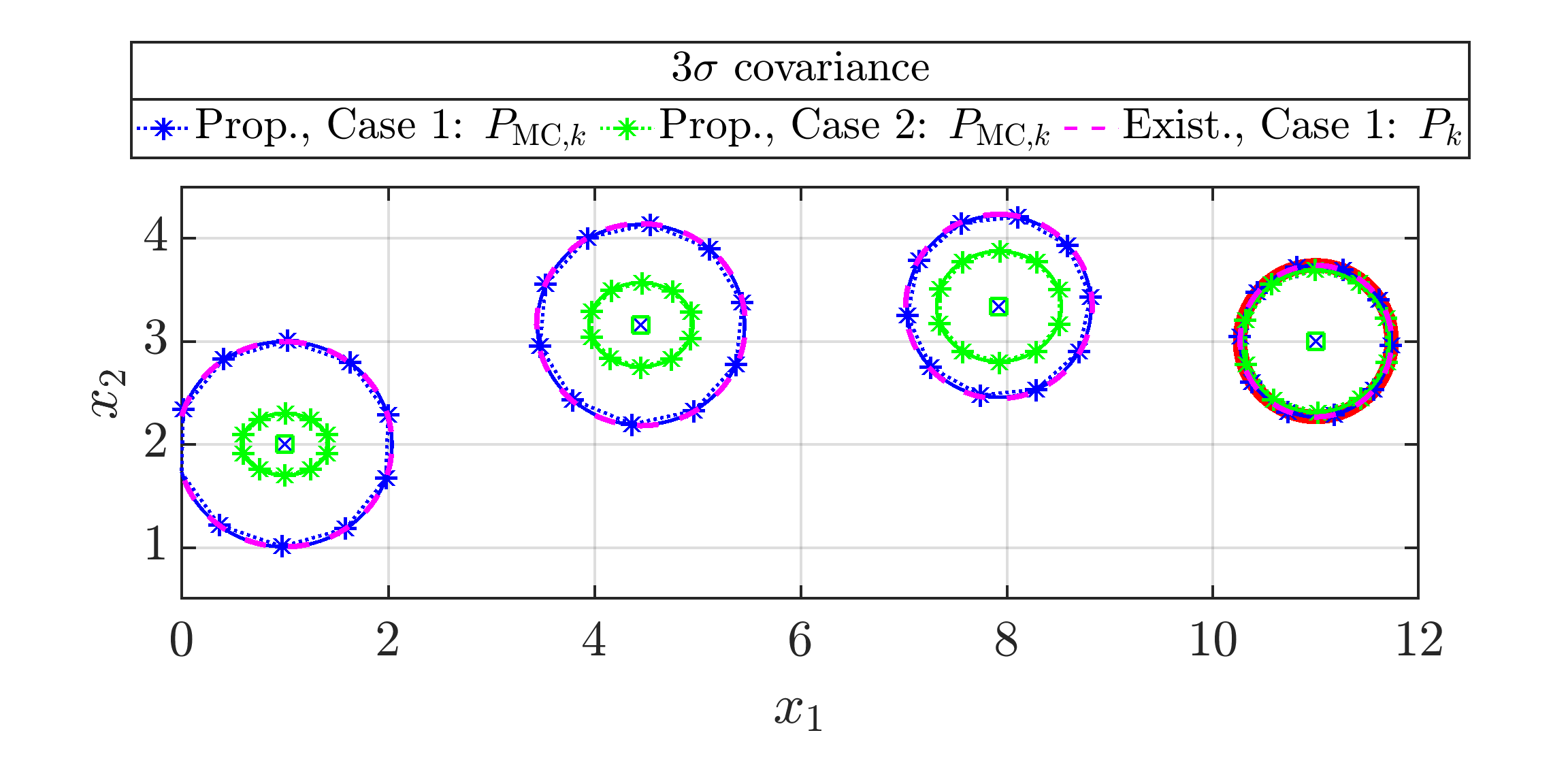}
	\caption{Validation of the proposed method for Cases 1 \& 2 at $k = 0, 5, 10, 19$ with the existing method \cite{Pilipovsky-CS-Output-Feedback} and a Monte Carlo simulation ($n_{\text{MC}} = 10{,}000$).}
	\label{fig: validation}
\end{figure}


\subsection{Comparison with Underweighted Kalman Gain}
This section demonstrates a scenario with a non-orthogonal estimator. The initial statistical relationship in Case 3 is orthogonal, except that an underweighted Kalman gain makes the statistics non-orthogonal after the first measurement update. 
\begin{equation}\label{eq: results uw}
\text{Case 3:}\quad
\mathbf{P}^-_0 = 
\begin{bmatrix}
    \hat{P}^-_0 + \widetilde{P}^-_0&\hat{P}^-_0\\
    \hat{P}^-_0 & \hat{P}^-_0
\end{bmatrix}
\quad
\text{with}
\quad
{L}_{\mathrm{uw},k}
\end{equation}

\begin{figure}[!tbh]
\centering\includegraphics[width=0.49\textwidth]{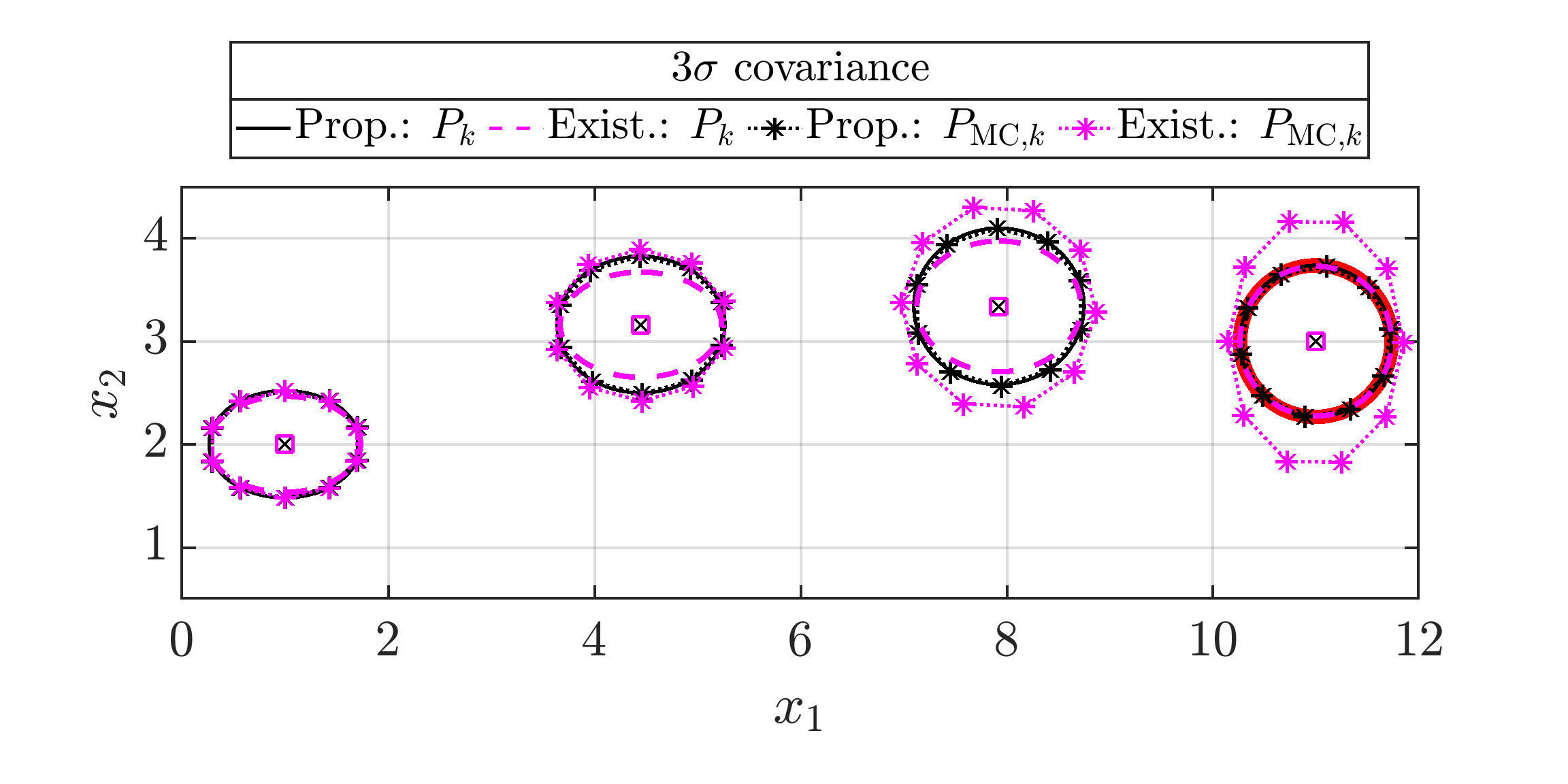}
	\caption{Case 3 with underweighted Kalman gain: true state distributions at $k = 0, 5, 10, 19$ calculated from both the proposed and existing \cite{Pilipovsky-CS-Output-Feedback} methods, compared with a Monte Carlo simulation ($n_{\text{MC}} = 10{,}000$).}
	\label{fig: case 3}
\end{figure}
Figure~\ref{fig: case 3} compares the predicted covariances with those obtained from a Monte Carlo simulation. The approach from \cite{Pilipovsky-CS-Output-Feedback} is unable to accurately capture the statistics of the true state because underweighting violates its required orthogonality condition. On the other hand, the framework proposed in this paper does not rely on the orthogonality assumption and is therefore able to produce solutions that are consistent with the actual statistical relationship between the true state and the state estimate. 

\subsection{Additional Analysis}
\begin{figure}[!tbh]
\centering\includegraphics[width=0.45\textwidth]{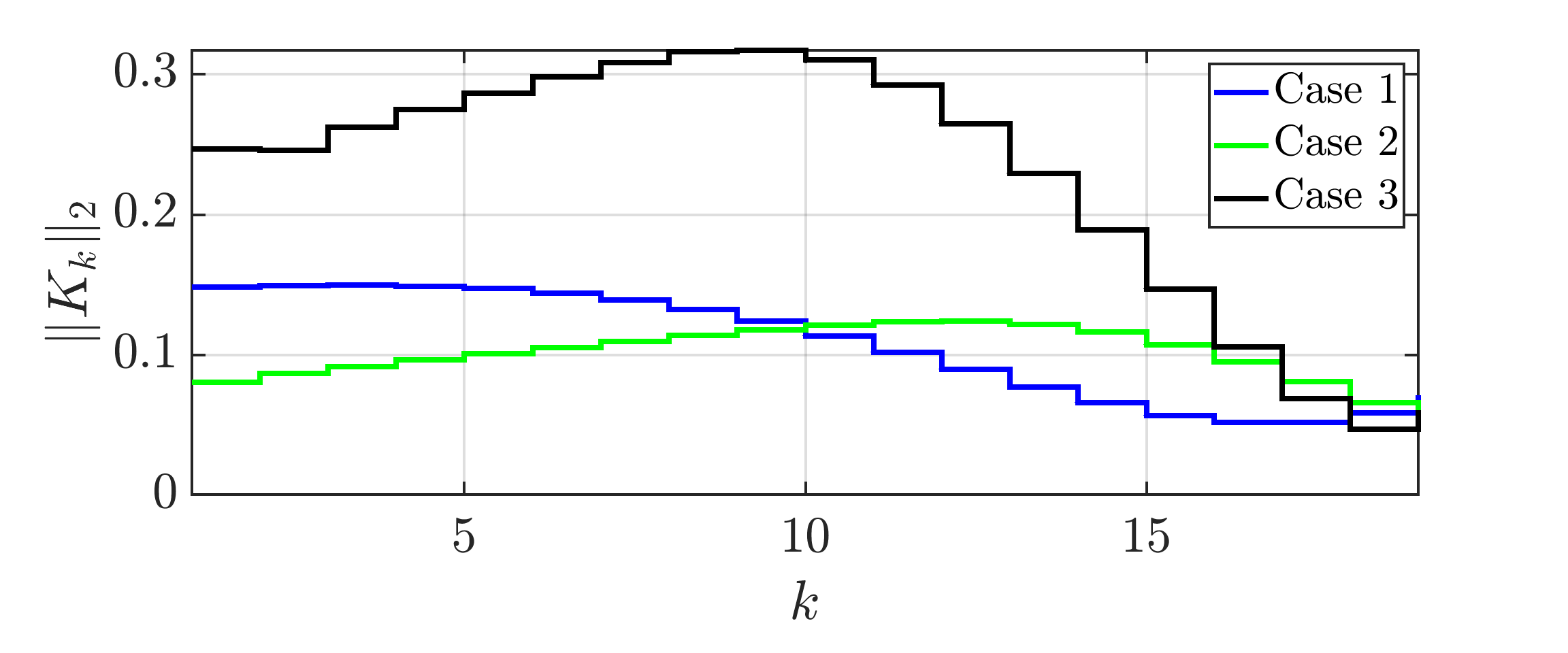}
	\caption{Time history for spectral norm of feedback gain matrix $K_k$ calculated using the proposed method.}
	\label{fig: control}
\end{figure}

Figure~\ref{fig: control} illustrates the magnitude of the feedback action applied over time. In Case 1, the feedback gain is largest at the beginning of the trajectory and decreases thereafter. In contrast, Case 2 exhibits a feedback profile that reaches its peak near the midpoint of the trajectory. This suggests that despite identical initial estimated state distributions and the policy's functional form, delaying the control actions may be optimal due to non-orthogonality. For Case 3, the suboptimality of the Kalman gain shows that control feedback peaks later in the trajectory compared to Case 1. These behaviors highlight that differences in the correlation structure can significantly influence the resulting control solution.

\begin{figure}[!tbh]
\centering\includegraphics[width=0.49\textwidth]{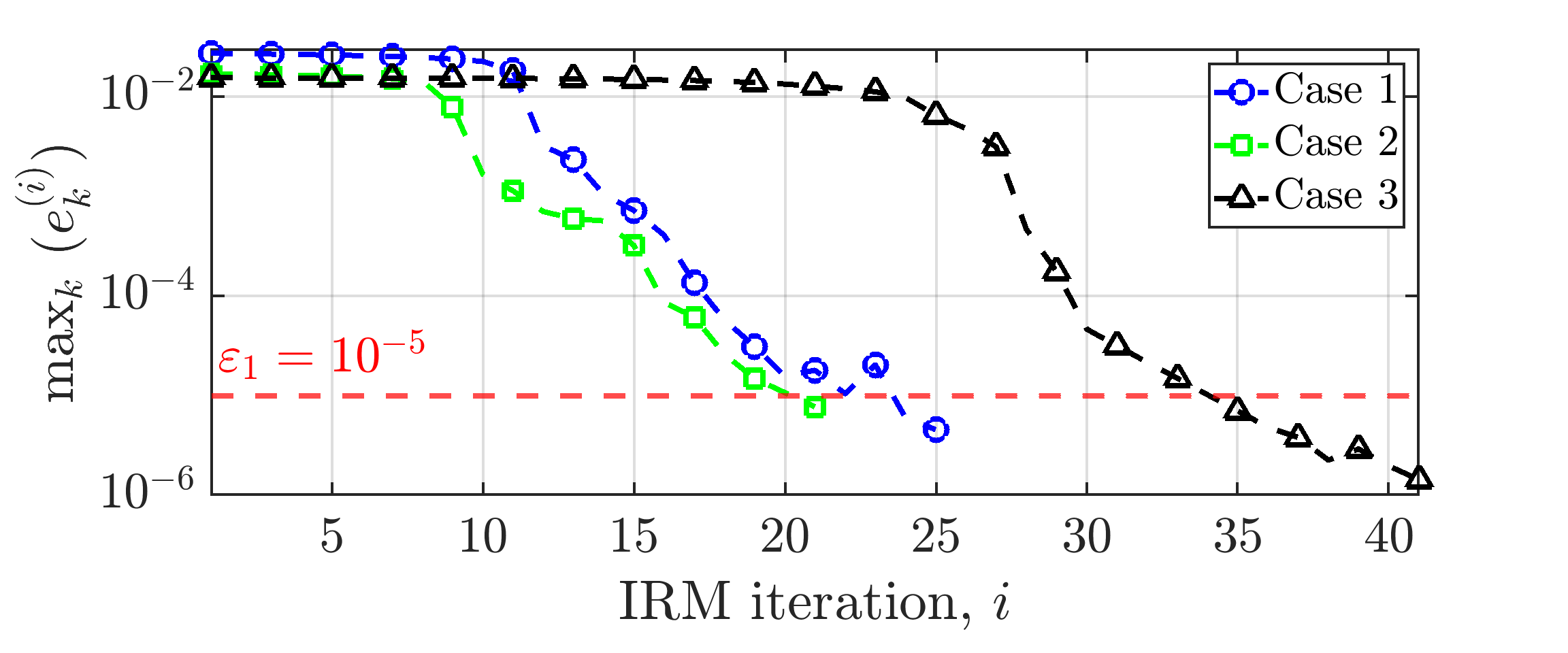}
	\caption{IRM convergence profile for each case.}
	\label{fig: erank iter}
\end{figure}

Figure~\ref{fig: erank iter} shows the IRM convergence profile of the three cases. This SCP problem converged after 25 iterations for Case 1 and 21 iterations for Case 2, and took less than one minute for each case. Case 3 converged after 41 iterations and took less than 2 minutes. In each case, the rank-constraint violations converge to below the tolerance, indicating that the rank condition is satisfied to the specified numerical accuracy and that the converged solution is feasible for the augmented-state CS formulation.


\section{Conclusions}
This paper generalizes the covariance steering with output feedback (CS-OF) problem to cases where the estimation error is not orthogonal to the state estimate. Existing CS-OF frameworks are limited by the orthogonality assumption and therefore cannot address scenarios with correlated estimation errors. The proposed approach casts the CS-OF problem as an augmented-state formulation to capture the coupled statistical relationships and applies a rank-constrained sequential convex programming framework for policy optimization. The resulting method is validated through comparisons with the prior approaches as well as large-scale Monte Carlo simulations. Although relaxing the orthogonality assumption introduces nonconvexity into the problem, the added modeling generality broadens the range of applications that can be addressed for CS-OF.

\section*{Appendix}
Appendix~\ref{sec: appendix} can be found in \cite{Qi-non-ortho}.

\bibliographystyle{IEEEtran}
\bibliography{references}
\newpage
\section{Appendix}\label{sec: appendix}
\subsection{Useful Identities}
\subsubsection{Covariance of Augmented Random Vectors}
Define random vectors $X \in \mathbb{R}^{n}$, $Y \in \mathbb{R}^{m}$, and $Z = [       X^\top, Y^\top]^\top\in \mathbb{R}^{n + m}$. The covariance matrix of $Z$ is
\begin{equation} \label{eq: general X and Y}
    \mathrm{Cov}(Z) = 
    \begin{bmatrix}
        \mathrm{Cov}(X) & \mathrm{Cov}(X,Y) \\
        \mathrm{Cov}(Y,X) & \mathrm{Cov}(Y)
    \end{bmatrix}
\end{equation}

If $X$ and $Y$ are uncorrelated (i.e., $\mathrm{Cov}(Y,X) = 0$),
\begin{equation}\label{eq: uncorr X and Y}
    \mathrm{Cov}\left(
    \begin{bmatrix}
        X+Y \\ Y
    \end{bmatrix}
    \right) = 
    \begin{bmatrix}
       \mathrm{Cov}(X) + \mathrm{Cov}(Y)& \mathrm{Cov}(Y) \\
        \mathrm{Cov}(Y) & \mathrm{Cov}(Y)
    \end{bmatrix}
\end{equation}

\subsubsection{Orthogonality and Unbiased Estimator}
Define true state $X\in \mathbb{R}^{n}$, estimate $\hat{X}\in \mathbb{R}^{n}$ and estimate error $\widetilde{X} = X - \hat{X}$. The error is orthogonal to the estimate $\hat{X}$ if $\mathbb{E}[\hat{X}\widetilde{X} ^\top] = 0$ and unbiased if $\mathbb{E}[\widetilde{X} ] = 0$. Under both conditions,
\begin{equation}\label{eq: filter cross corr}
    \mathrm{Cov}(\hat{X}, \hat{X}+\widetilde{X} ) = \mathrm{Cov}(\hat{X})
    \qquad
    \mathrm{Cov}(\hat{X}, \widetilde{X} ) =0
\end{equation}

\subsection{Proof of Lemma~\ref{lemma: CSOF equivalence}}\label{appendix: CSOF equivalence}
The claim is that if the orthogonality of the estimator is assumed, then Eqs.~\eqref{eq: covariance filter update} and \eqref{eq: covariance control update} reduce to Eqs.~\eqref{eq: CSOF, old version}. Using Eqs.~\eqref{eq: uncorr X and Y} and \eqref{eq: filter cross corr}, the a priori augmented covariance of an orthogonal estimator is:
\begin{equation}\label{eq: P_aug0 special case}
\mathbf{P}^-_k
=
\begin{bmatrix}
    \hat{P}^-_k + \widetilde{P}^-_k&\hat{P}^-_k\\
    \hat{P}^-_k&\hat{P}^-_k
\end{bmatrix}
\end{equation}
Using Eq.~\eqref{eq: P_aug0 special case} in Eq.~\eqref{eq: covariance filter update},
\begin{equation} \label{eq: appendix, filter update}
\begin{aligned}
\mathbf{P}_k 
&=
\begin{bmatrix}
  \hat{P}^-_k + \widetilde{P}^-_k & (L_kH_k\widetilde{P}^-_k + \hat{P}^-_{k})^\top \\ 
  L_kH_k\widetilde{P}^-_k + \hat{P}^-_{k} & \hat{P}^-_{k} + L_{k}P_{\widetilde{y}^-_{k}}L_{k}^\top \\ 
\end{bmatrix}
\end{aligned}
\end{equation}
Firstly, note $\hat{P}^-_{k+1} =(A_k +B_kK_k)\hat{P}_k(A_k +B_kK_k)^\top$. Next, given $L_k$ is the optimal Kalman gain in Eq.~\eqref{eq: optimal kalman gain}, Eq.~\eqref{eq: Kalman Covariance Update} can be rewritten using $\widetilde{P}_k = \widetilde{P}^-_k - L_k P_{\widetilde{y}^-_k} L_k^\top$. Also with orthogonality, $\hat{P}_k + \widetilde{P}_k = \hat{P}^-_k + \widetilde{P}^-_k$. Thus, Eq.~\eqref{eq: CSOF, old version} becomes
\begin{equation}\label{eq: appendix, prev work cov updates}
    \begin{aligned}
\hat{P}_{k} = \hat{P}^-_{k} + L_{k}P_{\widetilde{y}^-_{k}}L_{k}^\top, 
\qquad
P_k = \hat{P}^-_k + \widetilde{P}^-_k
    \end{aligned}
\end{equation}
where it can be seen that $P_k$ and $\hat{P}_k$ match with corresponding quantities in Eq.~\eqref{eq: appendix, prev work cov updates}. Eq.~\eqref{eq: filter cross corr} can be applied to Eq.~\eqref{eq: filter update} to verify that $\mathrm{Cov}(\hat{\boldsymbol{x}}_k,\boldsymbol{x}_k)$ in terms of the a priori values is equivalent to the cross-correlation term in Eq.~\eqref{eq: appendix, filter update}. This shows that  Eq.~\eqref{eq: covariance filter update} and \eqref{eq: covariance control update} produce the same statistical quantities as Eq.~\eqref{eq: CSOF, old version}.

\subsection{Simulation Parameters for Numerical Result}\label{appendix: sim params}
The dynamics of the system is given by
\begin{equation}
    A_k =
    \begin{bmatrix}
1 & 0 & \Delta t &0\\
0& 1& 0 &\Delta t \\
0&0&1&0\\
0&0&0&1\\
    \end{bmatrix}
\qquad
B_k =
    \begin{bmatrix}
\Delta t &0\\
0 &\Delta t \\
1&0\\
0&1\\
    \end{bmatrix}
\end{equation}
where $\Delta t = 0.2$. Let the linear system dynamics be subject to additive Gaussian process noise $G_k = 0.01 \cdot I_{n_x}$. The initial mean of $\boldsymbol{\mu}_0 = [1,2,3,2]^\top$ is used for the mean of the true state and the a priori estimate. Terminal mean and covariance constraint $\boldsymbol{\mu}_f = [11,3,0,0]^\top$ and $P_f = \mathrm{diag}([6,6,0.6,0.6]\cdot10^{-2})$ are imposed on the true state statistics. The measurement has $H_k = [0_{3\times1},I_3]$ with $D_k = \mathrm{diag}([0.1,0.1,0.1])$. The control weight is equal for both axes $R_k = I_2$ and $\mathrm{tr}(Z_k)$ is weighted by $\epsilon_Z = 10^{-3}$. For the IRM solver parameters, $w_0 = 1$, $\beta = 1.2$, and $\epsilon_1 = \epsilon_2= 10^{-5}$. The convex problems are solved via \texttt{CVX} with \texttt{MOSEK} in \texttt{MATLAB} R2024b.

For Case 1, $\widetilde{P}^-_0 = \mathrm{diag}([2,1,1.4,1.4]\cdot 10^{-2})$ and $\hat{P}^-_0= \mathrm{diag}([10,10,2,2]\cdot 10^{-2})$, and for Case 2, $\widetilde{P}^-_0 = \mathrm{diag}([8,9,0.6,0.6]\cdot 10^{-2})$ and $P_0= \mathrm{diag}([2,1,1.4,1.4]\cdot 10^{-2})$. Note that the value for the a priori estimate $\hat{P}^-_0$ is the same for both Case 1 and 2, but the initial estimation errors are different. 

For Case 3, $\widetilde{P}^-_0 = \mathrm{diag}([4,2,2.8,2.8]\cdot 10^{-2})$ and $\hat{P}^-_0 = \mathrm{diag}([2,1,1.4,1.4]\cdot 10^{-2})$. An underweighting factor of $p = 0.25$ is selected to accentuate the non-orthogonality. 

\subsection{Derivation of Initial Augmented Covariances}\label{appendix: case 1 and 2 initial covariance}
Let the filter's error $\widetilde{\boldsymbol{x}}^-_0=\boldsymbol{x}_0-\hat{\boldsymbol{x}}^-_0$ at the initial time. Note that there is no assumption here regarding the relationship between $\widetilde{\boldsymbol{x}}^-_k$ and $\boldsymbol{x}_k-\hat{\boldsymbol{x}}^-_k$ at future timesteps. Also, $\mathbb{E}[\widetilde{\boldsymbol{x}}^-_0]=0$, $\mathrm{Cov}(\widetilde{\boldsymbol{x}}^-_0) = \widetilde{P}^-_0$, and also $\mathrm{Cov}(-\widetilde{\boldsymbol{x}}^-_0) = \widetilde{P}^-_0$ also holds.

For Case 1, $\mathbb{E}[\boldsymbol{\hat{x}}^-_0(\widetilde{\boldsymbol{x}}^-_0)^\top] = 0$, so $\mathrm{Cov}(\boldsymbol{\hat{x}}^-_0,\widetilde{\boldsymbol{x}}^-_0) = 0$. It can be seen that $\boldsymbol{x}_0 = \boldsymbol{\hat{x}}^-_0 + \widetilde{\boldsymbol{x}}^-_0$, so from Eq.~\eqref{eq: uncorr X and Y}:
\begin{equation}
    \mathbf{P}^-_0 = 
\begin{bmatrix}
    \mathrm{Cov}(\hat{\boldsymbol{x}}^-_0) + \mathrm{Cov}(\widetilde{\boldsymbol{x}}^-_0) & \mathrm{Cov}(\hat{\boldsymbol{x}}^-_0)\\
    \mathrm{Cov}(\hat{\boldsymbol{x}}^-_0)& \mathrm{Cov}(\hat{\boldsymbol{x}}^-_0)
\end{bmatrix}
= 
\begin{bmatrix}
    \hat{P}^-_0 + \widetilde{P}^-_0&\hat{P}^-_0\\
    \hat{P}^-_0 & \hat{P}^-_0
\end{bmatrix}
\end{equation}

For Case 2, $\mathbb{E}[\boldsymbol{x}_0(\widetilde{\boldsymbol{x}}^-_0)^\top] =\mathbb{E}[\boldsymbol{x}_0(-\widetilde{\boldsymbol{x}}^-_0)^\top] = 0$, so $\mathrm{Cov}(\boldsymbol{x}_0,\widetilde{\boldsymbol{x}}^-_0) =\mathrm{Cov}(\boldsymbol{x}_0,-\widetilde{\boldsymbol{x}}^-_0)= 0$. It can be seen that $ \boldsymbol{\hat{x}}^-_0 = \boldsymbol{x}_0 + (- \widetilde{\boldsymbol{x}}^-_0)$, so similarly from Eq.~\eqref{eq: uncorr X and Y}:
\begin{equation}
    \mathbf{P}^-_0 = 
\begin{bmatrix}
    \mathrm{Cov}(\boldsymbol{x}_0) & \mathrm{Cov}(\boldsymbol{x}_0)\\
    \mathrm{Cov}(\boldsymbol{x}_0)& \mathrm{Cov}(\boldsymbol{x}_0) + \mathrm{Cov}(\widetilde{\boldsymbol{x}}^-_0)
\end{bmatrix}
= 
\begin{bmatrix}
    P_0 & P_0\\
    P_0& P_0 + \widetilde{P}^-_0
\end{bmatrix}
\end{equation}

\subsection{Monte Carlo Experiment}\label{appendix: MC experiment}
Table~\ref{tab: MC_Sampling} illustrates how the corresponding statistics would be interpreted from an experimental perspective via a Monte Carlo.

\begin{table}[H]
\caption{Procedure for Monte Carlo Sample}
\label{tab: MC_Sampling}

\renewcommand{\arraystretch}{1.5} 
\setlength{\tabcolsep}{6pt}       

\centering
\begin{tabular}{|c||c|c|}
\hline
Step & Case 1 & Case 2\\
\hline
1 & \multicolumn{2}{c|}{Sample $\widetilde{\boldsymbol{x}}^-_0$ from $N(0,\widetilde{P}^-_0)$} \\
\hline
2 & Sample $\hat{\boldsymbol{x}}^-_0$ from $N(\boldsymbol{\mu}_0,\hat{P}^-_0)$
& Sample $\boldsymbol{x}_0$ from $N(\boldsymbol{\mu}_0,P_0)$\\
\hline
3 & $\boldsymbol{x}_0 = \hat{\boldsymbol{x}}^-_0 + \widetilde{\boldsymbol{x}}^-_0$
& $\hat{\boldsymbol{x}}^-_0 = \boldsymbol{x}_0+\widetilde{\boldsymbol{x}}^-_0$ \\
\hline
\end{tabular}
\end{table}

\end{document}